\documentclass{article}
\usepackage{amsmath}
\usepackage{amssymb}
\usepackage{graphicx}
\usepackage{epstopdf}
\usepackage{amsthm}
\usepackage{hyperref}
\usepackage{amsfonts}
\usepackage{array}
\usepackage{amsthm}
\usepackage{tabularx}
\usepackage{subfigure}
\usepackage{caption}
\newtheorem{theorem}{Theorem}

\newtheorem{definition}{Definition}

\newtheorem{lemma}{Lemma}

\begin{document}
\title{$L_p$ and almost sure convergence of estimation on heavy tail index under random censoring}
\author{Yunyi Zhang, Jiazheng Liu, Zexin Pan, Dimitris N. Politis}
\maketitle
\abstract{
In this paper, we prove $L_p,\ p\geq 2$ and almost sure convergence of tail index estimator mentioned in \cite{grama2008} under random censoring and several assumptions. $p$th moment of the error of the estimator is proved to be of order $O\left(\frac{1}{\log^{m\kappa/2}n}\right)$ with given assumptions. We also perform several finite sample simulations to quantify performance of this estimator. Finite sample results show that the proposed estimator is effective in finding underlying tail index even when censor rate is high.
}
\section{Introduction and assumption}
\subsection{Introduction}
Research on heavy tail data is relevant to numerous statistical application, such as actuarial science \cite{doi:10.1080/10920277.1999.10595797}, economics \cite{doi:10.1198/073500101316970421} and etc. Tail index is one of the most crucial factor for long tail data since it is related with extreme quantiles of the underlying distribution, see \cite{STUPFLER20161} and \cite{10.1007/978-1-4612-3634-4_1} for further discussion. Hill \cite{10.2307/2958370} proposed an estimator for tail index and this estimator has been proved convergence under several assumptions and situations, we refer \cite{10.1007/978-1-4612-3634-4_13} and \cite{deheuvels_haeusler_mason_1988} as two examples. Apart from traditional Hill's estimator, Grama and Spokoiny \cite{grama2008} applied Kullback-Leibler divergence to estimate heavy tail index and proposed an estimator based on maximization local log-likeliihood method, Kratz and Resnick \cite{doi:10.1080/15326349608807407} proposed a qq type estimator and Politis et.al. \cite{Truncated_Estimator} proposed a truncated ratio statistics and proved its $L^p$ convergence.

On the other hand, it is common for dealing with incomplete observations, especially right censor data, in practical researches. Klein and Moeschberger \cite{Censor_Data} provide detail discussion and examples on this topic. For heavy tail data, censoring is more likely to occur. For example, in clinical trial, if the survival time obeys long tail distribution, it is more likely for the patient to survive after trial ends. Therefore, how to estimate heavy tail index under censoring is worth discussion. Beirlant and Guillou \cite{doi:10.1080/03461230152592764} proved consistency of a modified Hill's estimator under mild censoring, Einmahl et.al. \cite{einmahl2008} applied moment estimator in this problem and proved the asymptotic normality of the proposed estimator. Ndao et.al. \cite{NDAO201463} and Stupfler\cite{STUPFLER20161} generalized the result to the conditional heavy tail index.

Instead of convergence in probability, we mainly focus on almost sure convergence and $L^p$ convergence of estimator proposed in \cite{grama2008} for censoring data. Similar to \cite{einmahl2008} and \cite{Truncated_Estimator}, we apply a truncated version ratio type statistics for heavy tail index and use method proposed by Vasiliev \cite{Vasiliev2014} for proving convergence of the estimator.

We will give the basic assumption and statistics in \ref{Basic_Assumption}. In section \ref{L_p_Convergence} and \ref{As_Convergence}, we respectively discuss $L^p$ convergence and almost sure convergence of the proposed estimator, and numerical examples can be seen in section \ref{Simulation}. Finally, we make conclusions in section \ref{Conclusion}.
\subsection{Basic assumptions and main results \label{Basic_Assumption}}
In this part, we introduce basic assumptions, frequently used notations and the main statistics in this paper. The notations that are not listed below will be defined when being used.

Suppose $(X_1,Y_1),\ (X_2,Y_2),...,(X_n,Y_n)$ are i.i.d data from underlying distribution whose tail functions are respectively $P_X(x)=L_X(x)x^{-\alpha_X}$, $P_Y(x)=L_Y(x)x^{-\alpha_Y}$, and we further assume that $X_i,Y_i,\ i=1,2,...,n$ are mutually independent. Suppose the observed data are $(Z_i,\delta_i),\ Z_i=X_i\bigwedge Y_i=\min(X_i,Y_i)$ and $\delta_i=\mathbf{1}_{X_i\leq Y_i}$. Under this assumption, it is obvious that tail function of $Z_i$, $P_Z(x)=P_X(x)P_Y(x)=L_Z(x)x^{-\alpha_Z}$, here
\begin{equation}
L_Z(x)=L_X(x)L_Y(x),\ \alpha_Z=\alpha_X+\alpha_Y
\label{Z_dist}
\end{equation}
 Moreover, we suppose $L_X,\ L_Y$ are slow varying function (see \cite{Regular_Variation}).

According to Karamata's theorem\cite{Regular_Variation}, the slow varying function $L_K(x),\ K=X,Y$ satisfy
\begin{equation}
\log(L_K(x))=c_K(x)+\int_{a}^{x}\frac{\epsilon_K(y)}{y}dy,\ c_K(x)\to c,\ \epsilon_K\to 0
\label{Kara_Cond}
\end{equation}
as $x\to\infty$, and if we further assume that $L_K$ is differentiable, then formula \ref{Kara_Cond} intuitively implies that $c^{'}_K(x)\to 0$ and
\begin{equation}
\frac{xL_K^{'}(x)}{L_K(x)}=xc_K^{'}(x)+\epsilon_K(x)
\label{Deri_Kara}
\end{equation}
 Apart from the independent censor assumption, we hope the absolute value of derivative of $L_K,\ K=X,Y$ to be small enough as $x$ being sufficiently large so that the influence of $L_K$ on estimating tail index can be controlled by taking logarithm. This idea leads to assumption A2. The third assumption comes form \cite{Truncated_Estimator}.

Assumption A1: Suppose $X_i$ being i.i.d data and $Y_i$ being i.i.d censor time, $X_i,Y_i$ are mutually independent and respectively have tail function(that is, 1-cumulative distribution) $P_X(x)=L_X(x)x^{-\alpha_X}$, $P_Y(x)=L_Y(x)x^{-\alpha_Y}$. Thus the tail index of the data and censor time are $\gamma_X=1/\alpha_X,\ \gamma_Y=1/\alpha_Y$

Assumption A2: Suppose $L_K,\ K=X,Y$ are differentiable and there exists a number $\kappa>0$ such that, for $K=X,Y$, as $x\to\infty$.
\begin{equation}
\vert\frac{xL_K^{'}(x)}{L_K(x)}\vert=O\left(\frac{1}{\log^{\kappa} x}\right)
\end{equation}

Assumption A3: There exists a known constant $\gamma_0>0$ such that
\begin{equation}
\gamma_X\geq 2\gamma_0,\ \gamma_Y\geq 2\gamma_0
\end{equation}
From \ref{Z_dist} this implies that $\alpha_Z\leq 1/\gamma_0$ and corresponding $\gamma_Z\geq \gamma_0$

Table \ref{Nota_Fre} displays the frequently used notations and their meanings. In order to illustrate the main estimator, we first introduce several intermediate statistics.
\begin{table}
  \centering
  \caption{Frequently used notations}
  \begin{tabular}{l l}
  \hline\hline
  Notation & Meaning\\
  \hline
  $X_i,Y_i$ & $X_i,Y_i$ respectively represents the i.i.d real data and censor time\\
  \hline
  $(Z_i,\delta_i)$ & $Z_i=\min(X_i,Y_i)$ and $\delta_i=\mathbf{1}_{X_i\leq Y_i}$\\
  \hline
  $P_K(x),\ K=X,Y,Z$ & Tail function of random variable $X_i,Y_i,Z_i$\\
  \hline
  $L_K(x),\ K=X,Y,Z$ & Slow varying factor in the tail function\\
  \hline
  $f_K(x),\ K=X,Y,Z$ & Density function of random variable $K$\\
  \hline
  $\mathbf{1}_{h(K_1,..,.K_n)\in A},\ K=X,Y,Z$ & If $h(K_1,...,K_n)\in A$ the function equals 1 and 0 otherwise\\
  \hline
  $t(n)$ & $t(n)\to\infty$ as $n\to\infty$, see \ref{Condition_s_t} for definition\\
  \hline
  $s(n)$ & $s(n)\to 0$ as $n\to\infty$(see \ref{Condition_s_t})\\
  \hline
  $\gamma_K,\ K=X,Y,Z$ & Tail index for distribution $P_K$\\
  \hline
  $\alpha_K,\ K=X,Y,Z$ & $\alpha_K=1/\gamma_K$\\
  \hline
  $\vert\vert K\vert\vert_p$, & For random variable $K$, $\vert\vert K\vert\vert_p=(\mathbf{E}\vert K\vert^p)^{1/p}$ being $L_p$ norm\\
  \hline\hline
  \end{tabular}
  \label{Nota_Fre}
\end{table}
\begin{definition}[$\widehat p(x)$ and $\widehat q(x)$]
Suppose $x>0$ and sample size is $n$, then we respectively define $\widehat p(x)$ and $\widehat q(x)$ as
\begin{equation}
\widehat{p}(x)=\frac{1}{n}\Sigma_{i=1}^n \mathbf{1}_{Z_i\geq x}
\label{wide_p}
\end{equation}
and
\begin{equation}
\widehat{q}(x)=\frac{1}{n}\Sigma_{i=1}^n \delta_i\mathbf{1}_{Z_i\geq x}
\label{wide_q}
\end{equation}
\end{definition}
It is not difficult to see that $\widehat p(x)$ is estimator for $P_Z(x)$ and $\widehat q(x)$ is an estimator for $Prob(Z\leq x\cap X\leq Y)$. \ref{Moti_S} demonstrates the motivation for us to estimate this probability. Similar with \cite{einmahl2008}, the second estimator $\widehat \rho$ is applied to estimate $\frac{\gamma_Y}{\gamma_X+\gamma_Y}$.
\begin{definition}
With the notation in table \ref{Nota_Fre}, we define estimator $\widehat\rho$ as
\begin{equation}
\widehat{\rho}=\frac{\widehat{q}(t(n))}{\widehat{p}(t(n))}\mathbf{1}_{\widehat{p}\geq s(n)}
\end{equation}
Here $t(n)\to\infty$ and $s(n)\to 0$ satisfying
\begin{equation}
\begin{cases}
t(n)=n^\beta,\ s(n)=n^{-c},\ \beta<\frac{\gamma_0}{2},\ \frac{\beta}{\gamma_0}<c<1/2,\ \text{If A3 is satisfied}\\
t(n)=\log^\beta n,\ s(n)=n^{-c},\ \beta>0,\ 0<c<1/2\ \text{If A3 is not satisfied}
\end{cases}
\label{Condition_s_t}
\end{equation}
\label{Defff_2}
\end{definition}
We apply estimator defined in \cite{grama2008}, $\zeta$, for estimating tail index of the censored data $\gamma_Z$. According to \ref{Z_dist}, it is reasonable for expecting $\zeta$ to converge to $\frac{\gamma_X\gamma_Y}{\gamma_X+\gamma_Y}$.
\begin{definition}[Estimator for tail index $\gamma_Z$]
With the notation in table \ref{Nota_Fre}, we define the estimator for tail index $\gamma_Z$ as
\begin{equation}
\widehat \zeta=\frac{1}{\widehat{p}(t(n))}\int_{t(n)}^\infty\frac{\widehat{p}(y)}{y}dy\mathbf{1}_{\widehat p(t(n))\geq s(n)}=\frac{\mathbf{1}_{\widehat p(t(n))\geq s(n)}}{n\widehat p(t(n))}\left(\Sigma_{i=1}^n\log\left(\frac{Z_i}{t(n)}\right)\mathbf{1}_{Z_i\geq t(n)}\right)
\label{Stat_Fac}
\end{equation}
\end{definition}
Here we use convention that $\infty\times 0=0$.
Since $\gamma_X=\frac{\gamma_X\gamma_Y}{\gamma_X+\gamma_Y}/\frac{\gamma_Y}{\gamma_X+\gamma_y}$, estimator $\frac{\widehat \zeta}{\widehat \rho}$ is a candidate for estimating tail index of $X$. We will use a truncated version of this estimator. The key results of this paper is presented in theorem \ref{Theorem_K} and \ref{Theorem_AS}.

\begin{theorem}
Suppose A1, A2, A3 and choose $H_n=\frac{1}{\log\log n}$,$m\geq 4$, then we have
\begin{equation}
\vert\vert\frac{\widehat{\zeta}}{\widehat\rho}\mathbf{1}_{\widehat{\rho}\geq H_n}-\gamma_X\vert\vert_{m/2}=O\left(\frac{1}{\log^{\kappa}n}\right)
\end{equation}
\label{Theorem_K}
\end{theorem}

\begin{theorem}
Suppose condition A1-A3 and choose $H_n$ as in theorem \ref{Theorem_K}, then we have
\begin{equation}
\frac{\widehat\zeta}{\widehat\rho}\mathbf 1_{\widehat\rho\geq H_n}\to_{a.s.}\gamma_X
\end{equation}
\label{Theorem_AS}
\end{theorem}
\section{$L_p$ convergence of truncated statistics\label{L_p_Convergence}}
We first provide several crucial lemma that will be frequently used in the following proof.
\begin{lemma}
Suppose $X$ and $Y$ satisfy A1 and A2, then $Z=min(X,Y)$ has tail index $\gamma_Z=\frac{\gamma_X\gamma_Y}{\gamma_X+\gamma_Y}$, and $P_Z(x)=L_Z(x)x^{-1/\gamma_Z}$ with $L_Z$ satisfies A2
\label{Lemma_Zs}
\end{lemma}
\begin{proof}
Because of independents, we have, for arbitrary large $x$,
\begin{equation}
P_Z(x)=L_X(x)L_Y(x)x^{-\alpha_X-\alpha_Y}
\end{equation}
Therefore, $\gamma_Z=1/(\alpha_X+\alpha_Y)=\frac{\gamma_X\gamma_Y}{\gamma_X+\gamma_Y}$ and the first part is proved. For the second part, notice that
\begin{equation}
\vert\frac{xL_Z^{'}(x)}{L_Z(x)}\vert\leq \vert\frac{xL_X^{'}}{L_X}\vert+\vert\frac{xL_Y^{'}}{L_Y}\vert=O\left(\frac{1}{\log^\kappa n}\right)
\end{equation}
and the result is proved.
\end{proof}
The next one is introduced to provide a bound for the slow varying function.
\begin{lemma}
Suppose $L_K,\ K=X,Y$ satisfies condition A2, then for $\forall \epsilon>0$ being given, for sufficiently large $x$, we have
\begin{equation}
x^{-\epsilon}\leq L_K(x)\leq x^\epsilon,\ K=X,Y,Z
\end{equation}
\label{Lemma_Range}
\end{lemma}
\begin{proof}
This is equivalent as $\vert\log L_K(x)\vert\leq \epsilon\log x$ for large $x$. According to lemma \ref{Lemma_Zs}, $L_Z(x)$ also satisfies A2. We suppose $\kappa<1$ and if $\kappa\geq 1$, the derivative is of $o\left(\frac{1}{\log^{0.5} x}\right)$ for large $x$. Because of A2, there exists a constant $C>0$ and $x_0>0$ such that for arbitrary $x>x_0$,
\begin{equation}
\vert\log L_K(x)\vert\leq \vert\log L_K(x_0)\vert+\frac{C}{1-\kappa}(\log^{1-\kappa}x-\log^{1-\kappa}x_0)\leq\epsilon \log x
\end{equation}
for large $x$, and the result is proved.
\end{proof}
The third one involves a frequently used inequality.
\begin{lemma}
Suppose $K_i,\ i=1,2,...,n$ being i.i.d random variables and define $I_i=\mathbf{1}_{K_i\in A_n}$, here suppose $A_n$ is a Borel set with positive measure as a function of sample size $n$. Define $\widehat{r}=\frac{\Sigma_{i=1}^n I_i}{n}$ and $r=\mathbf{E} I_i$, and suppose $m\geq 2$ being a constant, then there exists a constant $C_m$ such that
\begin{equation}
\vert\vert\widehat{r}-r\vert\vert_m\leq \frac{C_m (r(1-r))^{1/m}}{\sqrt{n}}
\end{equation}
\label{Burk_Holder}
\end{lemma}
\begin{proof}
Define $f_k$ as
\begin{equation}
f_k=
\begin{cases}
\Sigma_{i=1}^k (I_i-r),\ k\leq n\\
\Sigma_{i=1}^n (I_i-r),\ k>n
\end{cases}
\end{equation}
Then, $f_k$ is a martingale. Since $m\geq 2$, from Minkowski inequality and Burkholder inequality \cite{10.2307/2959344}, we have
\begin{equation}
\begin{aligned}
n\vert\vert\widehat{r}-r\vert\vert_m\leq C_m\vert\vert\sqrt{\Sigma_{i=1}^n (I_i-r)^2}\vert\vert_m\\
=C_m\sqrt{\vert\vert\Sigma_{i=1}^n (I_i-r)^2\vert\vert_{m/2}}\\
\leq C_m\sqrt{\Sigma_{i=1}^n\vert\vert(I_i-r)^2\vert\vert_{m/2}}
\end{aligned}
\end{equation}
Since $m\geq 2$, we have $\mathbf{E}\vert I_i-r\vert^m\leq r(1-r)$ and the result is proved.
\end{proof}
Now, we start proving the $L_p$ convergence of estimator $\widehat{\rho}$.
\begin{theorem}
Suppose A1 and A2 and $m\geq 2$, then we have
\begin{equation}
\vert\vert\frac{\widehat{q}(t(n))}{\widehat{p}(t(n))}\mathbf{1}_{\widehat{p}(t(n))\geq s(n)}-\frac{\gamma_Y}{\gamma_X+\gamma_Y}\vert\vert_m=O\left(\frac{1}{s(n)\sqrt{n}}\right)+O\left(\frac{1}{P_Z(t(n))\sqrt{n}}\right)+O\left(\frac{1}{\log^\kappa t(n)}\right)
\end{equation}
\label{Theorem_rho}
\end{theorem}
\begin{proof}
According to \cite{Index_Random_Censoring}, we have that $\delta_i$ has the same distribution as $\mathbf{1}_{U_i\leq \lambda(Z_i)}$, here $U_i$ is uniform $[0,1]$ random variable being independent with $Z_i$ and
\begin{equation}
\lambda(x)=Prob(X\leq Y\vert Z=x)=\frac{f_X(x)P_Y(x)}{f_X(x)P_Y(x)+f_Y(x)P_X(x)}=\frac{1/\gamma_X-\frac{xL_X^{'}(x)}{L_X(x)}}{1/\gamma_X+1/\gamma_Y-\frac{xL_X^{'}(x)}{L_X(x)}-\frac{xL_Y^{'}(x)}{L_Y(x)}}
\label{Moti_S}
\end{equation}
Because of condition A2, for sufficiently large $x$, we have
\begin{equation}
\vert\lambda(x)-\frac{\gamma_Y}{\gamma_X+\gamma_Y}\vert\leq \frac{2}{(\gamma_X+\gamma_Y)^2}\left(\gamma_X\gamma_Y^2\vert\frac{xL_Y^{'}}{L_Y}\vert+\gamma^2_X\gamma_Y\vert\frac{xL_X^{'}}{L_X}\vert\right)
\end{equation}
Thus, there exists a constant $C$ such that for sufficiently large $x$, $\vert\lambda(x)-\frac{\gamma_Y}{\gamma_X+\gamma_Y}\vert\leq \frac{C}{\log^\kappa (x)}$ and correspondingly, from mean value theorem, we have
\begin{equation}
\vert \mathbf{E}\widehat{q}(t(n))-\frac{\gamma_Y P_Z(t(n))}{\gamma_X+\gamma_Y}\vert\leq \int_{t(n)}^\infty \vert\lambda(x)-\frac{\gamma_X}{\gamma_X+\gamma_Y}\vert f_Z(x)dx\leq \frac{C}{\log^\kappa t(n)}P_Z(t(n))
\label{Eq_range}
\end{equation}
From Minkowski inequality, we have
\begin{equation}
\begin{aligned}
\vert\vert\frac{\widehat{q}(t(n))}{\widehat{p}(t(n))}\mathbf{1}_{\widehat{p}(t(n))\geq s(n)}-\frac{\gamma_Y}{\gamma_X+\gamma_Y}\vert\vert_m\\
\leq \vert\vert\left(\frac{\widehat{q}(t(n))}{P_Z(t(n))}-\frac{\gamma_Y}{\gamma_X+\gamma_Y}\right)\frac{P_Z(t(n))}{\widehat{p}(t(n))}\mathbf{1}_{\widehat{p}(t(n))\geq s(n)}\vert\vert_m+\frac{\gamma_Y}{\gamma_X+\gamma_Y}\vert\vert\frac{P_Z(t(n))}{\widehat{p}(t(n))}\mathbf{1}_{\widehat{p}(t(n))\geq s(n)}-1\vert\vert_m
\end{aligned}
\end{equation}

If A3 is satisfied, then for sufficiently large $n$, from lemma \ref{Lemma_Zs} and \ref{Condition_s_t}, for $\forall \epsilon>0$
\begin{equation}
P_Z(t(n))\geq L_Z(t(n))t^{-1/\gamma_0}(n)\geq n^{-\beta(1/\gamma_0+\epsilon)}
\label{Css}
\end{equation}
choose small $\epsilon$ we have $c>\epsilon+\frac{\beta}{\gamma_0}$ and thus for large $n$, we have $P_Z(t(n))/2>s(n)$.

If A3 is not satisfied, from \ref{Condition_s_t}, similar with \ref{Css}, we have $P_Z(t(n))\geq \log^{-\beta(\epsilon+1/\gamma_0)} n>2s(n)$ for large $n$.

Thus from Chebyshev's inequality and lemma \ref{Burk_Holder}, we have
\begin{equation}
\begin{aligned}
\vert\vert\frac{P_Z(t(n))}{\widehat{p}(t(n))}\mathbf{1}_{\widehat{p}(t(n))\geq s(n)}-1\vert\vert_m\leq \vert\vert\left(\frac{P_Z(t(n))}{\widehat{p}(t(n))}-1\right)\mathbf{1}_{\widehat{p}(t(n))\geq s(n)}\vert\vert_m+Prob\left(\widehat{p}(t(n))<s(n)\right)^{1/m}\\
\leq \frac{\vert\vert\widehat{p}(t(n))-P_Z(t(n))\vert\vert_m}{s(n)}+\frac{2\vert\vert\widehat{p}(t(n))-P_Z(t(n))\vert\vert_m}{P_Z(t(n))}\\
=O\left(\frac{1}{s(n)\sqrt{n}}\right)+O\left(\frac{1}{P_Z(t(n))\sqrt{n}}\right)
\end{aligned}
\label{Cru_for_P_Z}
\end{equation}
This directly implies that $\vert\vert\frac{P_Z(t(n))}{\widehat{p}(t(n))}\mathbf{1}_{\widehat{p}(t(n))\geq s(n)}\vert\vert_m=O\left(1\right)$

For the first term, from Cauchy inequality, there exists a constant $C$ such that we have
\begin{equation}
\begin{aligned}
\vert\vert\left(\frac{\widehat{q}(t(n))}{P_Z(t(n))}-\frac{\gamma_Y}{\gamma_X+\gamma_Y}\right)\frac{P_Z(t(n))}{\widehat{p}(t(n))}\mathbf{1}_{\widehat{p}(t(n))\geq s(n)}\vert\vert_m \\ \leq
 \vert\vert\left(\frac{\widehat{q}(t(n))}{P_Z(t(n))}-\frac{\gamma_Y}{\gamma_X+\gamma_Y}\right)\vert\vert_{2m}\vert\vert\frac{P_Z(t(n))}{\widehat{p}(t(n))}\mathbf{1}_{\widehat{p}(t(n))\geq s(n)}\vert\vert_{2m}\\
\leq C \left(\frac{1}{P_Z(t(n))}\vert\vert\widehat{q}(t(n))-\mathbf{E}\widehat{q}(t(n))\vert\vert_{2m}+O\left(\frac{1}{\log^\kappa t(n)}\right)\right)
\end{aligned}
\end{equation}
From lemma \ref{Burk_Holder}, we get the result.
\end{proof}
Notice that, if we assume A3, form \ref{Condition_s_t}, we know that the convergence rate is of $O\left(\frac{1}{\log^\kappa n}\right)$, otherwise the convergence rate is of $O\left(1/\log^\kappa\log n\right)$.

In the next part, we will concentrate on estimating $\gamma_Z$. According to \cite{Index_Random_Censoring}, since $\gamma_Z=\frac{\gamma_X\gamma_Y}{\gamma_X+\gamma_Y}$, if we can find a suitable estimator $\widehat{\zeta}$ of $\gamma_Z$, since $\gamma_X=\gamma_Z/\frac{\gamma_Y}{\gamma_X+\gamma_Y}$, it is reasonable to consider $\widehat\zeta/\widehat\rho$.
We will prove $L_p$ convergence of its truncated version below. First we give a lemma.
\begin{lemma}
Suppose A1, A2, then as $t\to\infty$, we have
\begin{equation}
\frac{1}{P_Z(t)}\int_{t}^\infty\frac{P_Z(y)}{y}dy=\gamma_Z+O\left(\frac{1}{\log^\kappa t} \right)
\end{equation}
\label{Lemma_Expe}
\end{lemma}
\begin{proof}
Since $\gamma_Z=\frac{1}{P_Z(t)}\int_{t}^\infty\frac{L_Z(t)}{y^{(1+\alpha_Z)}}dy$, from Fubini-Tonelli theorem and A2, for large $t$, there exists constant $C$ such that
\begin{equation}
\begin{aligned}
\vert\frac{1}{P_Z(t)}\int_{t}^\infty\frac{P_Z(y)}{y}dy-\gamma_Z\vert=\frac{1}{P_Z(t)}\vert\int_{t}^\infty\frac{L_Z(y)-L_Z(t)}{y^{\alpha_Z+1}}dy\vert\\
\leq \frac{1}{P_Z(t)}\int_{t}^\infty\frac{dy}{y^{\alpha_Z+1}}\int_{t}^y \vert L_Z^{'}(z)\vert dz\\
=\frac{\gamma_Z}{P_Z(t)}\int_{t}^\infty \vert L_Z^{'}(z)\vert z^{-\alpha_Z}dz\\
\leq \frac{\gamma_Z C}{P_Z(t)}\int_{t}^\infty \frac{L_Z(z)}{\log^\kappa z}z^{-\alpha_Z-1}dz
\end{aligned}
\label{Thr}
\end{equation}
From mean value theorem, suppose $\alpha_Z=1/\gamma_Z$
\begin{equation}
\begin{aligned}
\int_{t}^\infty\frac{L_Z(z)}{\log^\kappa z}z^{-1/\gamma_Z-1}dz=\Sigma_{n=0}^\infty\int_{t\log^n t}^{t\log^{n+1}t}\frac{L_Z(z)}{z^{\alpha_Z+1}\log^\kappa z}dz\\
=\Sigma_{n=0}^\infty\frac{\gamma_Z L_Z(\eta_n)}{\log^\kappa \eta_n}\frac{1}{t^{\alpha_Z}\log^{n\alpha_Z}t}\left(1-\frac{1}{\log^{\alpha_Z} t}\right)\\
\leq \Sigma_{n=0}^\infty\frac{\gamma_Z L_Z(\eta_n)}{\left(\log t+n\log\log t\right)^\kappa}\frac{1}{t^{\alpha_Z}\log^{n\alpha_Z}t}
\end{aligned}
\label{Sec_Cr}
\end{equation}
Notice that, from assumption A2, if $\kappa\neq 1$, then
\begin{equation}
\begin{aligned}
\vert \log L_Z(\eta_n)-\log L_Z(t)\vert\leq \int_{t}^{\eta_n} \vert\frac{L_Z^{'}(z)}{L_Z(z)}\vert dz\leq \frac{C}{1-\kappa}(\log^{1-\kappa}\eta_n-\log^{1-\kappa}t)\\
\leq \frac{C}{\vert 1-\kappa\vert}\vert (\log t+(n+1)\log\log t)^{1-\kappa}-\log^{1-\kappa} t\vert
\end{aligned}
\label{Cru_lemma}
\end{equation}
And if $\kappa=1$, then similarly we have $\vert \log L_Z(\eta_n)-\log L_Z(t)\vert\leq $, here $C$ is a constant. We continue proof with 3 different cases.

Case 1: $0<\kappa<1$. In this case, for a given constant $D$ and sufficiently large $t$, equation \ref{Cru_lemma} is less than
\begin{equation}
\frac{C}{1-\kappa}\left(\frac{\log t}{\log^\kappa t}+\frac{((n+1)\log\log t)}{((n+1)\log\log t)^\kappa}-\log^{1-\kappa} t\right)\leq \frac{(n+1)^{1-\kappa}}{(1-\kappa)D}\log\log t
\end{equation}
This implies that $\frac{L_Z(\eta_n)}{L_Z(t)}\leq (\log t)^{\frac{(n+1)^{1-\kappa}}{D(1-\kappa)}}$, combine with \ref{Thr} and \ref{Sec_Cr}, we have
\begin{equation}
\vert\frac{1}{P_Z(t)}\int_{t}^\infty\frac{P_Z(y)}{y}dy-\gamma_Z\vert\leq C\gamma_z^2\left(\frac{L_Z(\eta_0)}{L_Z(t)\log^\kappa t}+\Sigma_{n=1}^\infty\frac{1}{(\log t+n\log\log t)^\kappa(\log t)^{n\alpha_Z-\frac{(n+1)^{1-\kappa}}{D(1-\kappa)}}}\right)
\end{equation}
Since
\begin{equation}
\log L_Z(\eta_0)-\log L_Z(t)\leq \frac{1}{1-\kappa}\left(\log^{1-\kappa} t\left(1+\frac{\log\log t}{\log t}\right)^{1-\kappa}-\log^{1-\kappa} t\right)\leq \frac{1}{1-\kappa} \frac{\log\log t}{\log^\kappa t}=o(1)
\end{equation}
Thus, for large $t$, $\frac{L_Z(\eta_0)}{L_Z(t)}<2$. Also, for $n\geq 1$, we have $n+1\leq 2n$ and
\begin{equation}
\frac{n\alpha_Z}{2}>\frac{(n+1)^{1-\kappa}}{D(1-\kappa)}\Leftarrow Dn^\kappa\alpha_Z>\frac{2^{2-\kappa}}{1-\kappa}\Leftarrow D\alpha_Z>\frac{2^{2-\kappa}}{1-\kappa}
\end{equation}
choose $D$ satisfies this condition then
\begin{equation}
\Sigma_{n=1}^\infty\frac{1}{(\log t)^{n\alpha_Z-\frac{(n+1)^{1-\kappa}}{D(1-\kappa)}}}\leq \Sigma_{n=1}^\infty \frac{1}{4^{n\alpha_Z-\frac{(n+1)^{1-\kappa}}{D(1-\kappa)}}}\leq \Sigma_{n=1}^\infty\frac{1}{2^{n\alpha_Z}}<\infty
\end{equation}
And we prove the result.

Case 2: $\kappa=1$. If $\kappa=1$, from \ref{Cru_lemma}, we have $\frac{L_Z(\eta_n)}{L_Z(t)}\leq \left(1+\frac{(n+1)\log\log t}{\log t}\right)^C$, correspondingly, combine with \ref{Thr} and \ref{Sec_Cr}, for large $t$, we have
\begin{equation}
\begin{aligned}
\vert\frac{1}{P_Z(t)}\int_{t}^\infty\frac{P_Z(y)}{y}dy-\gamma_Z\vert\leq \frac{C\gamma_Z^2}{\log t}\Sigma_{n=0}^\infty\frac{1}{\log^{n\alpha_Z}t}\left(1+\frac{(n+1)\log\log t}{\log t}\right)^C\\
\leq \frac{C\gamma_Z^2}{\log t}\left(2^C+\Sigma_{n=1}^\infty\frac{(n+2)^C}{2^{n\alpha_Z}}\right)
\end{aligned}
\end{equation}
Since $\Sigma_{n=1}^\infty \frac{(n+2)^C}{2^{n\alpha_Z}}<\infty$, the result is proved.

Case 3: $\kappa>1$. If $\kappa>1$, from \ref{Cru_lemma}, we have $\frac{L_Z(\eta_n)}{L_Z(t)}\leq \exp(2C/(\kappa-1))$, combine with \ref{Sec_Cr} and we prove the result.
\end{proof}
Now, we start to prove the $L_p$ convergence of statistics $\widehat\zeta$.
\begin{theorem}
Suppose A1,A2 and $m\geq 2$ then we have
\begin{equation}
\vert\vert\widehat\zeta-\gamma_Z\vert\vert_m=O\left(\frac{1}{s(n)\sqrt{n}}+\frac{1}{P_Z(t(n))\sqrt{n}}+\frac{1}{\log^\kappa t(n)}\right)
\end{equation}
\label{Theorem_zeta}
\end{theorem}
\begin{proof}
From definition of $\widehat\zeta$ (see \ref{Stat_Fac}) and Minkowski inequality,
\begin{equation}
\begin{aligned}
\vert\vert\widehat{\zeta}-\gamma_Z\vert\vert_m\leq \vert\vert\left(\frac{1}{P_Z(t(n))}\int_{t(n)}^\infty\frac{\widehat{p}(y)}{y}dy-\gamma_Z\right)\frac{P_Z(t(n))}{\widehat{p}(t(n))}\mathbf{1}_{\widehat{p}(t(n))\geq s(n)}\vert\vert_m+\\
\gamma_Z\vert\vert\left(\frac{P_Z(t(n))}{\widehat{p}(t(n))}-1\right)\mathbf{1}_{\widehat{p}(t(n))\geq s(n)}\vert\vert_m+\gamma_Z Prob\left(\widehat{p}(t(n))<s(n)\right)^{1/m}
\end{aligned}
\end{equation}
For the second and the third term, from \ref{Cru_for_P_Z} we know that these term is of order $O\left(\frac{1}{s(n)\sqrt{n}}+\frac{1}{P_Z(t(n))\sqrt{n}}\right)=o(1)$. For the first term, from Cauchy inequality and Minkowski inequality, it is less than
\begin{equation}
\vert\vert \frac{P_Z(t(n))}{\widehat{p}(t(n))}\mathbf{1}_{\widehat{p}(t(n))\geq s(n)}\vert\vert_{2m}\left(\frac{1}{P_Z(t(n))}\vert\vert\int_{t(n)}^\infty\frac{\widehat{p}(y)-P_Z(y)}{y}dy \vert\vert_{2m}+\vert\frac{1}{P_Z(t(n))}\int_{t(n)}^\infty\frac{P_Z(y)}{y}dy-\gamma_Z\vert\right)
\label{Div}
\end{equation}
From \ref{Cru_for_P_Z}, $\vert\vert \frac{P_Z(t(n))}{\widehat{p}(t(n))}\mathbf{1}_{\widehat{p}(t(n))\geq s(n)}\vert\vert_{2m}=O\left(1\right)$, from integral version Minkowski inequality and lemma \ref{Burk_Holder}, we have
\begin{equation}
\begin{aligned}
\vert\vert\int_{t(n)}^\infty\frac{\widehat{p}(y)-P_Z(y))}{y}dy \vert\vert_{2m}\leq \int_{t(n)}^\infty\frac{\vert\vert\widehat{p}(y)-P_Z(y)\vert\vert_{2m}}{y}dy\\
\leq \int_{t(n)}^\infty\frac{C_{2m}P_Z^{1/2m}(y)}{\sqrt{n}y}dy\\
=\int_{t(n)}^\infty\frac{C_{2m}L_Z^{1/2m}(y)}{\sqrt{n}y^{1+\alpha_Z/2m}}dy
\end{aligned}
\end{equation}
From lemma \ref{Lemma_Range}, choose $\epsilon=\alpha_Z/2$, for sufficiently large $y$, $L_Z(y)^{1/2m}\leq y^{\alpha_Z/4m}$ and thus the integration is less than $\frac{4mC_{2m}}{\alpha_Z\sqrt{n}}t(n)^{-\alpha_Z/4m}=o(1/\sqrt{n})$. Thus,
\begin{equation}
\frac{1}{P_Z(t(n))}\vert\vert\int_{t(n)}^\infty\frac{\widehat{p}(y)-P_Z(y)}{y}dy \vert\vert_{2m}=O\left(\frac{1}{P_Z(t(n))\sqrt{n}}\right)
\label{Ref_for_1}
\end{equation}
For the second term in \ref{Div}, use lemma \ref{Lemma_Expe} and we prove the result.
\end{proof}
Similarly, if in addition we assume A3, from \ref{Condition_s_t} we know that the convergence rate is of $O\left(\frac{1}{\log^\kappa n}\right)$ and otherwise the convergence rate becomes $O\left(\frac{1}{\log^\kappa\log n}\right)$.

Finally, we apply discussions above to prove theorem \ref{Theorem_K}.

\begin{proof}[Proof for theorem \ref{Theorem_K}]
We choose $\mu=\nu=m/2$ in theorem 1 of $\cite{Vasiliev2014}$, according to \ref{Theorem_rho} and \ref{Theorem_zeta}, since exists constant $C$ such that for sufficiently large $n$,
\begin{equation}
\begin{aligned}
\mathbf{E}\vert\widehat{\zeta}-\gamma_Z\vert^{2\nu}=\vert\vert \widehat{\zeta}-\gamma_Z\vert\vert_m^m\leq \frac{C^m}{\log^{m\kappa} n}\\
w_n=\mathbf{E}\vert\widehat{\rho}-\frac{\gamma_Y}{\gamma_X+\gamma_Y}\vert^{2\mu}\leq \frac{C^m}{\log^{m\kappa} n}
\end{aligned}
\end{equation}
Choose $\phi_n(m)=\frac{2^{m-1}}{\left(\gamma_Y/(\gamma_X+\gamma_Y)\right)^m}\frac{C^m(\gamma_Z^m+\left(\frac{\gamma_X}{\gamma_X+\gamma_Y}\right)^m)}{\log^{m\kappa} n}=O\left(\frac{1}{\log^{m\kappa}n}\right)$ and choose $\beta=m/4$, $\beta\geq 1$, then we have
\begin{equation}
\mathbf{E}\vert\frac{\widehat\zeta}{\widehat\rho}\mathbf{1}_{\widehat\rho\geq H_n}-\gamma_X\vert^{2\beta}\leq V_n(\beta)
\label{F_1}
\end{equation}
Here
\begin{equation}
V_n(\beta)=O\left(\phi_n(2\beta)+\frac{\phi_n^{1/2}(m)w_n^{1/4}}{H_n^\beta}+\frac{\phi_n^{1/2}(m)w_n^{1/2}}{H_n^{2\beta}}+w_n\right)=O\left(\frac{1}{\log^{m\kappa/2} n}\right)
\label{F_2}
\end{equation}
Combine \ref{F_1} and \ref{F_2}, we prove the result.
\end{proof}
In particular, this directly proves the $L_{m/2}$ convergence of the statistics.
\section{Almost sure convergence of tail index estimator\label{As_Convergence}}
In this section, we try to prove the almost sure convergence of the tail index estimator under assumption A1-A3. We first introduce two lemma.
\begin{theorem}
Suppose A1-A3, and $t(n), s(n)$ are chosen as in \ref{Condition_s_t}, then we have
\begin{equation}
\frac{\widehat{q}(t(n))}{\widehat{p}(t(n))}\mathbf{1}_{\widehat{p}(t(n))\geq s(n)}\to_{a.s.}\frac{\gamma_Y}{\gamma_X+\gamma_Y}
\end{equation}
\label{As_p_q}
\end{theorem}
\begin{proof}
From Borel-Cantelli lemma \cite{Probabilty_Stochastic}, it suffices to show that, for $\forall \epsilon>0$,
\begin{equation}
\Sigma_{n=1}^\infty Prob\left(\vert\frac{\widehat{q}(t(n))}{\widehat p(t(n))}\mathbf{1}_{\widehat p(t(n))\geq s(n)}-\frac{\gamma_Y}{\gamma_X+\gamma_Y}\vert> 3\epsilon\right)<\infty
\label{Key_pq}
\end{equation}
Since
\begin{equation}
\begin{aligned}
Prob\left(\vert\frac{\widehat q(t(n))}{\widehat p(t(n))}\mathbf{1}_{\widehat p(t(n))\geq s(n)}-\frac{\gamma_Y}{\gamma_X+\gamma_Y}\vert>3\epsilon\right)\\
\leq Prob\left(\vert\left(\frac{\widehat q(t(n))}{P_Z(t(n))}-\frac{\mathbf{E}\widehat q(t(n))}{P_Z(t(n))}\right)\left(\frac{P_Z(t(n))}{\widehat p(t(n))}-1\right)\mathbf{1}_{\widehat p(t(n))\geq s(n)}\vert\geq \epsilon\right)\\
+Prob\left(\vert\frac{\mathbf{E}\widehat q(t(n))}{P_Z(t(n))}\frac{P_Z(t(n))}{\widehat p(t(n))}\mathbf{1}_{\widehat p(t(n))\geq s(n)}-\frac{\gamma_Y}{\gamma_X+\gamma_Y}\vert\geq \epsilon\right)\\
+Prob\left(\vert\frac{\widehat q(t(n))-\mathbf{E}\widehat q(t(n))}{P_Z(t(n))}\vert\mathbf{1}_{\widehat p(t(n))\geq s(n)}\geq \epsilon\right)
\end{aligned}
\end{equation}
We will separately discuss these 3 terms below.

For the first term, notice that for $\forall k>1$, from mean value inequality and Minkowski inequality, we have
\begin{equation}
\begin{aligned}
Prob\left(\vert\left(\frac{\widehat q(t(n))-\mathbf E\widehat q(t(n))}{P_Z(t(n))}\right)\left(\frac{P_Z(t(n))}{\widehat p(t(n))}-1\right)\vert\mathbf{1}_{\widehat p(t(n))\geq s(n)}\geq \epsilon\right)\\
\leq \frac{1}{\epsilon^k}\vert\vert\left(\frac{\widehat q(t(n))-\mathbf E \widehat q(t(n))}{P_Z(t(n))}\right)\left(\frac{P_Z(t(n))}{\widehat p(t(n))}-1\right)\mathbf 1_{\widehat p(t(n))\geq s(n)}\vert\vert_k^k\\
\leq \frac{1}{2^k\epsilon^k}\vert\vert\left(\frac{\widehat q(t(n))-\mathbf E\widehat q(t(n))}{P_Z(t(n))}\right)^2+\left(\frac{\widehat p(t(n))-\mathbf E \widehat p(t(n))}{\widehat p(t(n))}\right)^2\mathbf 1_{\widehat p(t(n))\geq s(n)}\vert\vert_k^k\\
\leq \frac{1}{2^k\epsilon^k}\left(\vert\vert\frac{\widehat q(t(n))-\mathbf E\widehat q(t(n))}{P_Z(t(n))}\vert\vert_{2k}^2+\frac{1}{s(n)^2}\vert\vert\widehat p(t(n))-\mathbf E\widehat p(t(n))\vert\vert_{2k}^2\right)^k
\end{aligned}
\label{F_p}
\end{equation}
From lemma \ref{Burk_Holder}, $\vert\vert\widehat q(t(n))-\mathbf E \widehat q(t(n))\vert\vert_{2k}^2=O\left(\frac{P_Z^{1/k}(t(n))}{n}\right)$ and $\vert\vert\widehat p(t(n))-\mathbf E\widehat p(t(n))\vert\vert_{2k}^2=O\left(\frac{P_Z^{1/k}(t(n))}{n}\right)$, choose $k>\max\left(\frac{1}{1/2-\beta/\gamma_0},\frac{1}{1-2c}\right)$ then the convergence of summation of first term is proved.

For the second term, notice that it is smaller than
\begin{equation}
Prob\left(\vert\frac{\mathbf E\widehat q(t(n))}{P_Z(t(n))}\frac{P_Z(t(n))}{\widehat p(t(n))}\mathbf{1}_{\widehat p(t(n))\geq s(n)}-\frac{\mathbf E\widehat q(t(n))}{P_Z(t(n))}\vert\geq \epsilon-\vert\frac{\mathbf E\widehat q(t(n))}{P_Z(t(n))}-\frac{\gamma_Y}{\gamma_X+\gamma_Y}\vert\right)
\end{equation}
From \ref{Eq_range}, for sufficiently large $n$, $\vert\frac{\mathbf E\widehat q(t(n))}{P_Z(t(n))}-\frac{\gamma_Y}{\gamma_X+\gamma_Y}\vert\leq \epsilon/2$, and $\frac{\mathbf E\widehat q(t(n))}{P_Z(t(n))}\leq \frac{2\gamma_Y}{\gamma_X+\gamma_Y}$. Since for $\forall k>1$
\begin{equation}
\begin{aligned}
Prob\left(\vert\frac{P_Z(t(n))}{\widehat p(t(n))}\mathbf 1_{\widehat p(t(n))\geq s(n)}-1\vert\geq \frac{(\gamma_X+\gamma_Y)\epsilon}{4\gamma_Y}\right)\\
\leq \left(\frac{4\gamma_Y}{\epsilon(\gamma_X+\gamma_Y)}\right)^k\left(\frac{1}{s(n)}\vert\vert\widehat p(t(n))-\mathbf E\widehat p(t(n))\vert\vert_k+Prob(\widehat p(t(n))<s(n))\right)^k
\end{aligned}
\label{X_s}
\end{equation}
From \ref{Cru_for_P_Z} and similar to \ref{F_p}, choose sufficiently large $k$ and we know that summation of this term converges.

For the third term, similar with \ref{F_p} and we prove can prove the convergence of summation. Since \ref{Key_pq} is true, almost sure convergence is proved as well.
\end{proof}

\begin{theorem}
Suppose A1-A3 and suppose $s(n)$ and $t(n)$ are chosen as in \ref{Condition_s_t}, then we have
\begin{equation}
\widehat\zeta\to_{a.s.}\gamma_Z
\end{equation}
definition of $\widehat\zeta$ is in \ref{Stat_Fac}.
\label{As_zeta}
\end{theorem}
\begin{proof}
From Borel-Cantelli \cite{Probabilty_Stochastic} lemma, it suffices to show that, for $\forall \epsilon>0$,
\begin{equation}
\Sigma_{n=1}^\infty Prob\left(\vert\widehat \zeta-\gamma_Z\vert\geq 4\epsilon\right)<\infty
\label{bc_zeta}
\end{equation}
Since the above term is less than
\begin{equation}
\vert\widehat\zeta-\frac{1}{P_Z(t(n))}\int_{t(n)}^\infty\frac{\widehat p(y)}{y}dy\vert+\vert\frac{1}{P_Z(t(n))}\int_{t(n)}^\infty\frac{\widehat p(y)-P_Z(y)}{y}dy\vert+\vert\frac{1}{P_Z(t(n))}\int_{t(n)}^\infty \frac{P_Z(y)}{y}dy-\gamma_Z\vert
\end{equation}
According to lemma \ref{Lemma_Expe}, for sufficiently large $n$, $\vert\frac{1}{P_Z(t(n))}\int_{t(n)}^\infty\frac{P_Z(y)}{y}dy-\gamma_Z\vert<2\epsilon$. Thus, there exists a constant $n_0$ such that
\begin{equation}
\begin{aligned}
\Sigma_{n=n_0}^\infty Prob\left(\vert\widehat \zeta-\gamma_Z\vert\geq 4\epsilon\right)\\
\leq \Sigma_{n=n_0}^\infty Prob\left(\vert\widehat\zeta-\frac{1}{P_Z(t(n))}\int_{t(n)}^\infty\frac{\widehat p(y)}{y}dy\vert\geq\epsilon\right)+Prob\left(\frac{1}{P_Z(t(n))}\vert\int_{t(n)}^\infty\frac{\widehat p(y)-P_Z(y)}{y}dy\vert\geq \epsilon\right)
\end{aligned}
\label{ft_zeta}
\end{equation}
For the second term, notice for arbitrary $k>1$, from \ref{Ref_for_1}
\begin{equation}
\begin{aligned}
Prob\left(\frac{1}{P_Z(t(n))}\vert\int_{t(n)}^\infty\frac{\widehat p(y)-P_Z(y)}{y}dy\vert\geq \epsilon\right)\leq \frac{1}{\epsilon^k}\left(\frac{1}{P_Z(t(n))}\vert\vert\int_{t(n)}^\infty\frac{\widehat p-P_Z(y)}{y}dy\vert\vert_k\right)^k\\
=O\left(\frac{1}{P_Z^k(t(n))n^{k/2}}\right)
\end{aligned}
\label{Ys}
\end{equation}
Choose $k>\frac{2}{\frac{1}{2}-\frac{\beta}{\gamma_0}}$ we prove the convergence of summation for the second term.

The first term of \ref{ft_zeta} is less than
\begin{equation}
\begin{aligned}
Prob\left(\frac{\int_{t(n)}^\infty\frac{P_Z(y)}{y}dy}{P_Z(t(n))}\vert\frac{P_Z(t(n))\mathbf 1_{\widehat p(t(n))\geq s(n)}}{\widehat p(t(n))}-1\vert\geq \frac{\epsilon}{2}\right)\\
+Prob\left(\frac{1}{P_Z(t(n))}\vert\left(\frac{P_Z(t(n))\mathbf 1_{\widehat p(t(n))\geq s(n)}}{\widehat p(t(n))}-1\right)\left(\int_{t(n)}^\infty \frac{\widehat p(y)-P_Z(y)}{y}dy\right)\vert\geq \frac{\epsilon}{2}\right)
\end{aligned}
\end{equation}
According to lemma \ref{Lemma_Expe}, For large $n$, $\int_{t(n)}^\infty \frac{P_Z(y)}{y}dy/P_Z(t(n))\leq 2\gamma_Z$, so for sufficiently large $n$,
\begin{equation}
\begin{aligned}
Prob\left(\frac{\int_{t(n)}^\infty\frac{P_Z(y)}{y}dy}{P_Z(t(n))}\vert\frac{P_Z(t(n))\mathbf 1_{\widehat p(t(n))\geq s(n)}}{\widehat p(t(n))}-1\vert\geq \frac{\epsilon}{2}\right)\\
\leq Prob\left(\frac{\vert P_Z(t(n))-\widehat p(t(n))\vert}{s(n)}\geq \frac{\epsilon}{4\gamma_Z}\right)+Prob\left(\widehat p(t(n))<s(n)\right)
\end{aligned}
\end{equation}
From \ref{X_s} we know the convergence of summation on this term.

Also, from mean value inequality,
\begin{equation}
\begin{aligned}
Prob\left(\frac{1}{P_Z(t(n))}\vert\left(\frac{P_Z(t(n))\mathbf 1_{\widehat p(t(n))\geq s(n)}}{\widehat p(t(n))}-1\right)\int_{t(n)}^\infty \frac{\widehat p(y)-P_Z(y)}{y}dy\vert\geq \frac{\epsilon}{2}\right)\\
\leq Prob\left(\frac{1}{P_Z^2(t(n))}\vert\int_{t(n)}^\infty\frac{\widehat p(y)-P_Z(y)}{y}dy\vert^2\geq \frac{\epsilon}{2}\right)\\
+ Prob\left(\left(\frac{P_Z(t(n))\mathbf 1_{\widehat p(t(n))\geq s(n)}}{\widehat p(t(n))}-1\right)^2\geq\epsilon/2\right)
\end{aligned}
\end{equation}
From \ref{Ys} and \ref{X_s} we know the convergence of summation. Thus, \ref{bc_zeta} is proved and the almost sure convergence is also proved.
\end{proof}

Finally, we prove the almost sure convergence of $\frac{\widehat \zeta}{\widehat \rho}\mathbf 1_{\widehat \rho\geq H_n}$.
\begin{proof}[Proof for theorem \ref{Theorem_AS}]
According to theorem \ref{As_p_q} and \ref{As_zeta}, since $\gamma_X,\gamma_Y>0$, we have
\begin{equation}
\frac{\widehat \zeta}{\widehat \rho}\to_{a.s.}\gamma_X
\end{equation}
Since $H_n\to 0$ as $n\to\infty$, according to theorem \ref{As_p_q}, there exists a $n_0>0$ such that
\begin{equation}
\begin{aligned}
\Sigma_{n=n_0}^\infty \mathbf 1_{\widehat \rho<H_n}\leq \Sigma_{n=n_0}^\infty \mathbf 1_{\vert\widehat \rho-\frac{\gamma_Y}{\gamma_X+\gamma_Y}\vert>\frac{\gamma_Y}{2(\gamma_X+\gamma_Y)}}<\infty
\end{aligned}
\end{equation}
which means that $\mathbf 1_{\widehat \rho\geq H_n}\to_{a.s.} 1$. Thus, the product of these two terms converges to $\gamma_X$.
\end{proof}
\section{Simulations and numerical examples\label{Simulation}}
In this section, we suppose $X_i$ and $Y_i,\ i=1,2,...n$ obey log gamma distribution, whose density is
\begin{equation}
f_K(x)=C_{f_K}x^{-\alpha_K-1}\log^{\beta_K-1} x,\ K=X,Y,\ x\geq 1,\ \alpha_K>0
\label{loggamma}
\end{equation}
We first prove that this distribution satisfies assumption A2.
\begin{theorem}
Distribution with density \ref{loggamma} satisfies condition A2
\end{theorem}
\begin{proof}
Slow varying part of distribution \ref{loggamma} is
\begin{equation}
L_K(x)=C_{f_K}x^{\alpha_K}\int_{x}^\infty y^{-\alpha_K-1}\log^{\beta_K-1}y\ dy
\end{equation}
Since
\begin{equation}
\alpha_K\int_x^\infty y^{-\alpha_K-1}\log^{\beta_K-1}y\ dy=x^{-\alpha_K}\log^{\beta_K-1}x+(\beta_K-1)\int_x^\infty y^{-\alpha_K-1}\log^{\beta_K-2}y \ dy
\end{equation}
We have
\begin{equation}
\begin{aligned}
\vert\frac{xL_K^{'}(x)}{L_K(x)}\vert=\frac{\vert(\alpha_K\int_x^\infty y^{-\alpha_K-1}\log^{\beta_K-1}y\ dy)-x^{-\alpha_K}\log^{\beta_K-1}x \vert}{\int_x^\infty y^{-\alpha_K-1}\log^{\beta_K-1}y\ dy}\\
=\frac{\vert\beta_K-1\vert\int_x^\infty y^{-\alpha_K-1}\log^{\beta_K-2}y\ dy}{\int_x^\infty y^{-\alpha_K-1}\log^{\beta_K-1}y\ dy}
\end{aligned}
\end{equation}
Suppose $h_K(x)=\int_x^\infty y^{-\alpha_K-1}\log^{\beta_K-1}y\ dy$, since $h_K$ is decreasing and $h_K(\infty)=0$, then
\begin{equation}
\frac{\int_x^\infty \frac{-h_K^{'}(y)}{\log y}\ dy}{h_K(x)}\leq \frac{1}{\log x}
\end{equation}
and thus assumption A2 is satisfied with $\kappa=1$.
\end{proof}
Figure \ref{Figure_1} to \ref{Figure_3} demonstrates the performance of estimator $\frac{\widehat\zeta}{\widehat\rho}\mathbf 1_{\widehat\rho\geq H_n}$ under different conditions. Parameters we choose for simulation is listed in table \ref{Simu_Para}, sample size is assumed to be 10000 for case 1-5 and 50000 for case 6. We use relative error
\begin{equation}
\delta=\frac{\vert\frac{\widehat{\zeta}}{\widehat\rho}\mathbf{1}_{\widehat{\rho}\geq H_n}-\gamma_X\vert}{\gamma_X}
\end{equation}
to evaluate finite sample performance of our estimator. We perform 50 times numerical experiments and the error bars in figure \ref{Figure_1} to \ref{Figure_3} show the maximum, minimum and average relative error under different $t(n)$. Following definition \ref{Defff_2}, $t(n)=n^\beta$, $\beta$ coincides with notation Beta in figure \ref{Figure_1}-\ref{Figure_3}. As we can see,

1. If tail index of censor time is less than the underlying data, performance of tail index estimator will be inferior.

2. Choosing suitable $t(n)$ is critical for making tail index estimator reliable. Choosing too small or too big $t(n)$ leads to increase of relative error.

3. For suitable $t(n)$, tail index estimator has good performance even when censor rate is high.
\begin{table}
  \centering
  \caption{Parameter for simulation, definition of $\beta_K,K=X,Y$ see \ref{loggamma}}
  \begin{tabular}{c l l l l l l}
  \hline\hline
  Case & $\gamma_X$ & $\gamma_Y$ & $\beta_X$ & $\beta_Y$ & $\gamma_0$ & Average censor rate\\
    1 &  2.0     &    2.0     &    1.2    &   1.4     &    0.2     &          0.441\\  \hline
    2 &  1.0     &    2.0     &    0.5    &   0.5     &    0.3     &          0.392\\  \hline
    3 & 1.0     &    2.0     &    1.5    &   1.5     &    0.3     &          0.291\\  \hline
    4 & 0.5     &    0.476   &    1.0    &   1.0     &    0.1     &          0.512\\  \hline
    5 & 0.5     &    0.4     &    1.0    &   1.0     &    0.1     &          0.555\\  \hline
    6 & 0.5     &    0.4     &    1.0    &   1.0     &    0.1     &          0.556\\
  \hline\hline
  \end{tabular}
  \label{Simu_Para}
\end{table}
\begin{figure}
\subfigure{
\begin{minipage}{0.5\linewidth}
\includegraphics[width=2.5in]{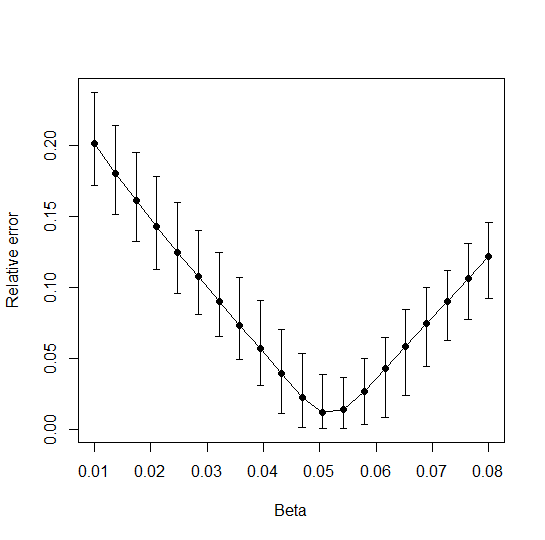}
\end{minipage}
}
\subfigure{
\begin{minipage}{0.5\linewidth}
\includegraphics[width=2.5in]{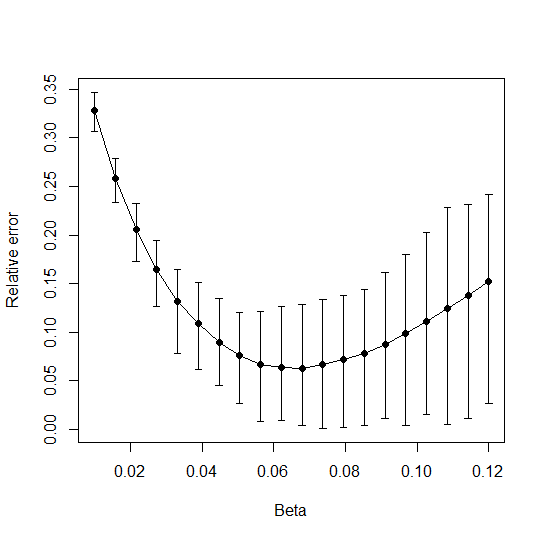}
\end{minipage}
}
\caption{Numerical experiment case 1(left) and 2(right), sample size is 10000, parameters used in these cases coincide with table \ref{Simu_Para}}
\label{Figure_1}
\end{figure}
\begin{figure}
\subfigure{
\begin{minipage}{0.5\linewidth}
\includegraphics[width=2.5in]{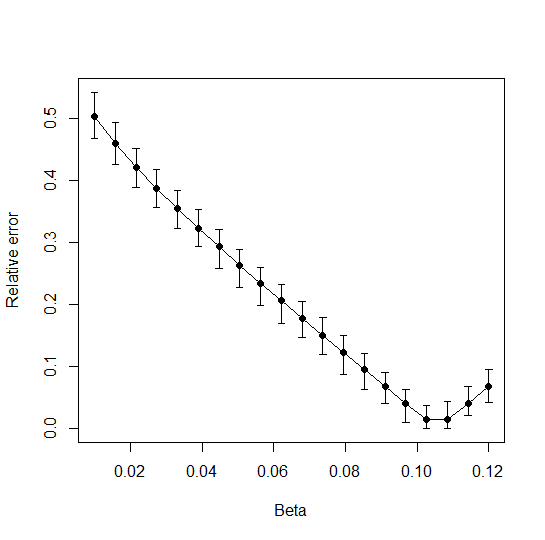}
\phantomcaption
\end{minipage}
}
\subfigure{
\begin{minipage}{0.5\linewidth}
\includegraphics[width=2.5in]{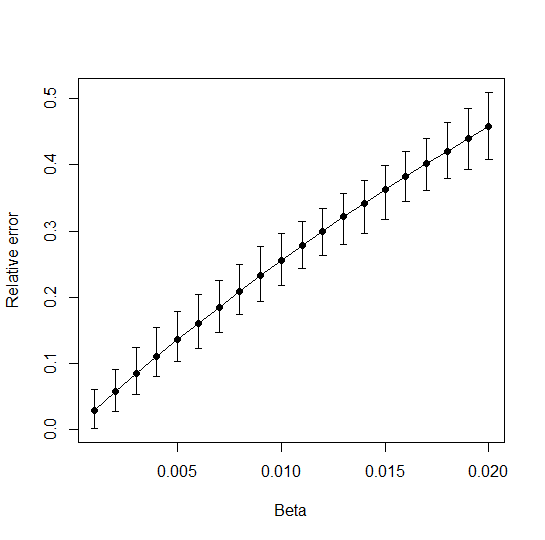}
\phantomcaption
\end{minipage}
}
\addtocounter{figure}{-2}
\caption{Numerical experiment case 3(left) and 4(right), sample size is 10000, parameters used in these cases coincide with table \ref{Simu_Para}}
\label{Figure_2}
\end{figure}
\begin{figure}
\subfigure{
\begin{minipage}{0.5\linewidth}
\includegraphics[width=2.5in]{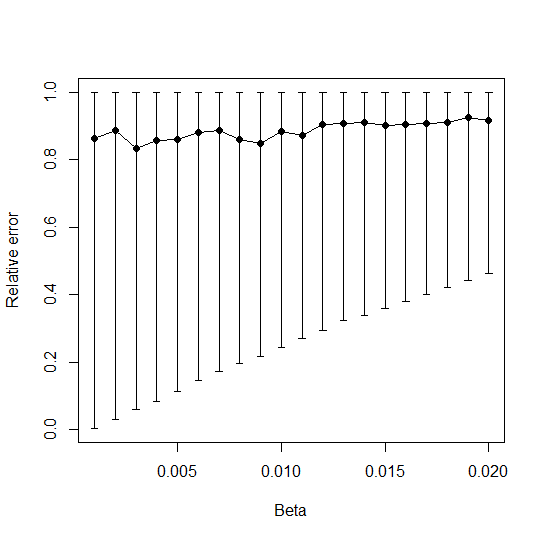}
\phantomcaption
\end{minipage}
}
\subfigure{
\begin{minipage}{0.5\linewidth}
\includegraphics[width=2.5in]{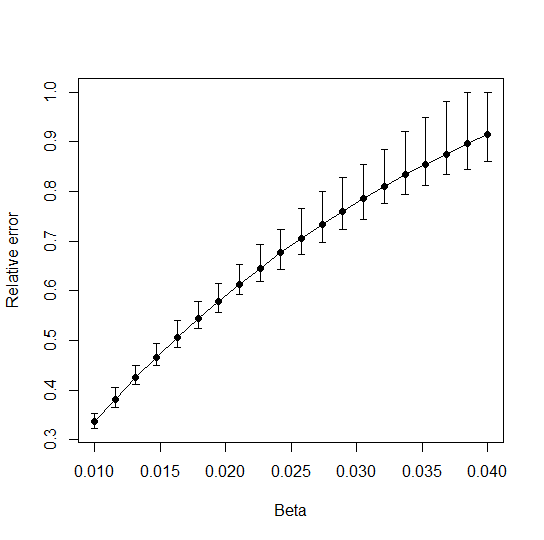}
\phantomcaption
\end{minipage}
}
\addtocounter{figure}{-2}
\caption{Numerical experiment case 5(left) and 6(right), sample size is 10000 for case 5 and 50000 for case 6, parameters used in these cases coincide with table \ref{Simu_Para}}
\label{Figure_3}
\end{figure}
\section{Conclusion\label{Conclusion}}
In this paper, we focus on proving almost sure convergence and $L_p$ convergence of estimator provided by Grama and Spokoiny \cite{grama2008} under random censoring and condition A1-A3. We also perform numerical experiments with data satisfying log gamma distribution. Numerical results demonstrate the usefulness of our tail index estimator when sample size is finite.
\bibliographystyle{unsrt}
\bibliography{Ref_for_Tail_Index}
\end{document}